\documentclass[sn-mathphys]{sn-jnl}

\jyear{2022}%

\theoremstyle{thmstyleone}%
\newtheorem{theorem}{Theorem}[section]
\newtheorem{lemma}{Lemma}[section]
\newtheorem{corollary}{Corollary}[section]

\theoremstyle{thmstyletwo}%
\newtheorem{remark}{Remark}[section]%
\newtheorem{problem}{Problem}[section]
\theoremstyle{thmstylethree}%
\newtheorem{definition}{Definition}[section]%

\usepackage[normalem]{ulem}
\newcommand{\bb}[1]{\textcolor{blue}{#1}}

\newcommand{\ignore}[1]{}

\begin{document}

\title[A Fast Randomized Algorithm for Computing an Approximate Null Space]{A Fast Randomized Algorithm for Computing an Approximate Null Space}


\author*[1]{\fnm{Taejun} \sur{Park}}\email{park@maths.ox.ac.uk}

\author[1]{\fnm{Yuji} \sur{Nakatsukasa}}\email{nakatsukasa@maths.ox.ac.uk}

\affil*[1]{\orgdiv{Mathematical Institute}, \orgname{University of Oxford}, \orgaddress{\street{Woodstock Road} \city{Oxford}, \postcode{OX2 6GG}, \country{UK}}}


\abstract{Randomized algorithms in numerical linear algebra can be fast, scalable and robust. This paper examines the effect of sketching on the right singular vectors corresponding to the smallest singular values of a tall-skinny matrix. We analyze a fast algorithm 
by Gilbert, Park and Wakin for finding the trailing right singular vectors using randomization by examining the quality of the solution using multiplicative perturbation theory. For an $m\times n$ ($m\geq n$) matrix, the algorithm runs with complexity $O(mn\log n +n^3)$ which is faster than the standard $O(mn^2)$ methods. In applications, numerical experiments show great speedups including a $30\times$ speedup for the AAA algorithm and $10\times$ speedup for the total least squares problem.}

\keywords{Multiplicative perturbation theory, null space, randomized algorithm, sketching, singular subspace}


\pacs[MSC Classification]{65F20, 65F55, 15A18, 15A42}

\maketitle

\section{Introduction}
\label{sec:intro}
The right singular vector(s) corresponding to the smallest singular value(s) is used for finding the null space of a tall-skinny matrix $A$, solving total least squares problems \cite{tls80,tls91} and finding a rational approximation via the AAA algorithm \cite{aaa}, among others. In this work, we study the randomized algorithm from \cite{exist} by analyzing the accuracy of the result using multiplicative perturbation theory developed by Li \cite{li2}. 

Randomized algorithms in numerical linear algebra have proved to be very useful, improving speed and scalability of linear algebra problems \cite{hmt,mt20}. In particular, the sketching techniques have been shown to be a powerful tool in problems such as low-rank approximation, least squares problems, linear systems and eigenvalue problems \cite{nt21,woodruff14}. In particular, for low-rank approximation problems, randomized SVD \cite{cw09,hmt,naka20,tropp17} has been very successful. They give a near-optimal solution to the low-rank matrix problems at a lower complexity than traditional methods. These algorithms focus on approximating the top few singular values and their corresponding singular vectors. On the other hand, relatively little attention has been paid to the \emph{bottom} singular values and their corresponding singular vectors. These are often needed in problems such as total least squares \cite{tls80,tls91}. In particular, the right singular vector(s) corresponding to the zero singular value span the null space. Also, the singular vector corresponding to the smallest singular value minimizes the norm of the error or residual, for example, in total least squares problems \cite{tls80} and the AAA algorithm for rational approximation \cite{aaa}.

In this work, we will focus on the smallest (few) singular value(s) and their corresponding right singular vector(s). For a tall-skinny matrix $A\in \mathbb{R}^{m\times n}$, they are the solution to the following optimization problem
\begin{equation} \label{nsprob}
    \min_{V^*V=I_k} \norm{AV}_\text{F}.
\end{equation}
The standard method for computing the null space or the last right singular vector corresponding to the smallest singular value is to use the SVD or rank-revealing factorizations such as RRQR factorizations \cite{rrqr87,rrqr92}, which costs $O(mn^2)$ flops for an $m\times n$ matrix $A$ with $m\geq n$. There are other methods such as the TSQR \cite{TSQR} and Cholesky QR which has the same complexity $O(mn^2)$, but can be improved with parallelization. In addition, Cholesky QR uses the Gram matrix which squares the condition number. The sketch-and-solve method in this paper will give us a theoretical flop count of $O(mn\log n+n^3)$ and possibly even lower when the matrix is structured, e.g. sparse. In this method we left-multiply the original matrix $A\in \mathbb{C}^{m\times n}$ with $m\geq n$ by a random sketching matrix $S\in \mathbb{C}^{s\times m}$ where $m\geq s\geq n$. The integer $s$ is called the sketch size. Typically we have $m\gg s = cn$ where $c>1$ is a modest constant, say $c=2$ or $c=4$. We then work with the matrix $SA \in \mathbb{C}^{s\times n}$ by taking its SVD, and finding its trailing singular vectors. Since $SA$ is a smaller-sized matrix, which contains a compressed information of $A$ (See Sections \ref{sec:problem} and \ref{sec:RPT}), $SA$ can act as a good substitute for $A$ in some settings, for example, if $A$ is too large to fit in memory or when $A$ is streamed. This makes the computation more efficient in terms of both speed and storage. The question of course is to examine the quality of singular vectors obtained this way.

The sketch-and-solve method is usually only useful if the sketching matrix preserves geometry in the sense that the norm of every vector in a span of $A$ is approximately preserved under sketching \cite{mt20,woodruff14}, i.e., 
$1-\epsilon \leq \norm{SAx}_2/\norm{Ax}_2 \leq 1+\epsilon$ for some $0<\epsilon<1$, for every nonzero $x\in \mathbb{C}^n$. A large class of sketching matrices are known to preserve geometry under mild conditions, including Gaussian matrices \cite{mp}, subsampled randomized trignometric transforms (SRTT) \cite{hmt,tropp11}, CountSketch \cite{cw13,woodruff14} and 1-hashed randomized Hadamard transform (H-RHT) \cite{cartis21}. Different sketching matrices have different requirements for the size of the sketch $s$. For example, we require $s = \Theta( n\log n)$ for SRTTs while for Gaussian matrices and H-RHT we require the optimal $s = \Theta(n)$ to ensure that geometry is preserved with high probability. The details about the failure probability and the conditions under which theoretical guarantees can be achieved can be found in the papers cited above. 

For the sketch-and-solve method, the quality of the solution needs to be assessed. Since subspace embeddings preserve norms from the original space, it is straightforward to see (Theorem~\ref{thm:4timesres}) that the residual norm in Equation \eqref{nsprob} for the sketch-and-solve method will be on the same order as the residual norm of the actual solution. However, this does not immediately imply that the the two solution vectors are close to each other. In this work, we will assess the quality of the solution by deriving bounds for the sine of the angle between the sketched solution and the original solution. Specifically, we will quantify this bound using multiplicative perturbation theory from \cite{li2} by Li. The perturbation is multiplicative because the original matrix gets multiplied by a sketching matrix rather than undergoing an additive perturbation. This is different from the classical perturbation theory \cite{dk70,wedin1972}, which is additive. The classical result scales poorly for small singular values when the perturbation is close to a unitary matrix, whereas the multiplicative perturbation theory can overcome this issue.

\subsection{Existing work and Contribution}
Gilbert, Park and Wakin 
 \cite{exist} discuss the same algorithm where they analyze the accuracy of the singular values and the right singular vectors obtained using the sketch $SA$. They use the 2-norm error to quantify the accuracy of the right singular vectors obtained this way, whereas we use the canonical angles which is arguably more natural in this setting. Furthermore, we extend the analysis to comparing subspaces of either the same or different dimensions (Section \ref{sec:RPT}). This is particularly useful if we are extracting a subspace from the sketch $SA$ rather than a single vector, for example, when we are solving total least squares problems with multiple right-hand sides (Section \ref{sec:app}). It is also important to mention that if we want to extract the right singular subspace of dimension larger than one corresponding to a multiple singular value (for example, singular values corresponding to zero in the case of computing the null space of a matrix) then the bound used to quantify the accuracy of a single singular vector in the right singular subspace is useless as the gap is zero, but when we compare the subspace as a whole, we can get meaningful bound as will be shown in Section \ref{sec:RPT}. Moreover, when the gap between the target singular value and the rest is small, it is well known that the corresponding target singular vector is ill-conditioned~\cite{wedin1972}, because the condition number of \bb{computing} a singular vector is inversely proportional to the gap. In such cases it is often difficult to obtain nontrivial bounds for the accuracy; however, we show that useful bounds can be obtained by examining the angle between the target vector(s) and a computed subspace of different (larger) dimension.

\subsection{Notation}
Throughout, the symbol $^*$ is used to denote the (conjugate) transpose of a vector or a matrix. We use $\norm{\cdot}$ for a unitarily invariant norm, $\norm{\cdot}_2$ for the spectral norm or the vector-$\ell_2$ norm and $\norm{\cdot}_\text{F}$ for the Frobenius norm. Unless specified otherwise $\sigma_i(A)$ denotes the $i$th largest singular value of the matrix $A$. Lastly, we use MATLAB style notation for matrices and vectors. For example, for the $k$th to $(k+j)$th columns of a matrix $A$ we write $A(:,k:k+j)$.


\section{Problem statement and the algorithm}
\label{sec:problem}
Let us formally define the problem.
\begin{problem}[Null Space Problem] \label{problem}
Let $A\in \mathbb{C}^{m\times n}$ where $m\geq n$ and $k\in \mathbb{Z}^+$. Find $W\in \mathbb{C}^{n\times k}$ that solves the following optimization problem
\begin{equation*}
    \min_{V^*V=I_k} \norm{AV}_\mathrm{F}.
\end{equation*}
\end{problem}
This problem statement does not find the null space of $A$ exactly; indeed the null space of $A$ is usually the trivial $0$, but to consider a broader class of problems we will call this problem the null space problem. The solution, $W$ in the problem statement, to the null space problem is the trailing $k$ right singular vectors, that is, the $k$ right singular vectors corresponding to the smallest $k$ singular values. $W$ is unique up to reordering of columns when the corresponding singular values are distinct from each other and also distinct from the other $(n-k)$ singular values. The standard way to calculate this is by computing the SVD and extracting the trailing $k$ right singular vectors. This costs $O(mn^2)$ flops. For $m\gg n$, computing the SVD can become expensive, so we use sketching matrices to get a near-optimal solution with a lower complexity.

The most widely used class of sketching matrices for analysis is Gaussian random matrices, whose entries are independent standard normal random variables. However, Gaussian matrices are not always the most efficient choice, so other sketching matrices such as the SRFT matrices \cite{hmt} are often used in practice. Here we discuss the Gaussian and the SRFT matrices. For Gaussian matrices we have the following.
\begin{theorem}[Mar$\check{\text{c}}$enko and Pastur \cite{mp}, Davidson and Szarek \cite{ds01}] \label{MP}
Consider an $s\times n$ Gaussian random matrix $G$ with $s \geq n$. Then
\begin{equation*}
    \sqrt{s}-\sqrt{n} \leq \mathbb{E}[\sigma_{\min}(G)] \leq \mathbb{E}[\sigma_{\max}(G)] \leq \sqrt{s}+\sqrt{n}.
\end{equation*}
Furthermore, for any $t>0$,
\begin{equation*}
    \max\Big\{\mathbb{P}\big(\sigma_{\max}(G)\geq \sqrt{s}+\sqrt{n}+t\big), \mathbb{P}\big(\sigma_{\min}(G)\leq \sqrt{s}-\sqrt{n}-t\big)\Big\} \leq \exp{(-t^2/2)}
\end{equation*}
\end{theorem}
This theorem implies that rectangular Gaussian matrices with aspect ratio $s/n>1$, that is, more rows than columns, are well-conditioned with singular values that lie in $[\sqrt{s}-\sqrt{n}-t,\sqrt{s}+\sqrt{n}+t]$ with failure probability that decreases exponentially with $t$. Theorem \ref{MP} can be used to see that
\begin{equation*}
    \frac{\sigma_i(GA/\sqrt{s})}{\sigma_i(A)} = O(1)
\end{equation*} holds with high probability \cite{mt20} where $A$ is an $m\times n$ matrix and $G$ is an $s\times m$ Gaussian matrix with $m\geq s\geq n$. In other words, sketching approximately preserves the singular values of the original matrix. 

The cost for applying a Gaussian sketch to an $m\times n$ matrix is $O(mn^2)$ operations, which has the same order as most classical numerical linear algebra algorithms. This makes Gaussian sketches not very useful in practice. 

There is an analogous result for SRFT matrices \cite{hmt,tropp11} which require a slightly larger sketch size and come with failure probability that is higher than Gaussian matrices. An SRFT matrix is an $m\times s$ matrix with $m\geq s$ of the form $\Omega = \sqrt{\frac{m}{s}}DFR^*$ where $D$ is a random $m\times m$ diagonal matrix whose entries are independent and take $\pm 1$ with equal probability, $F$ is the $m\times m$ unitary discrete Fourier transform and $R$ is a random $s \times m$ matrix that restricts an $m$-dimensional vector to $s$ coordinates chosen uniformly at random. The analogous result is from \cite{hmt}.
\begin{theorem} \label{srft}
Let $U$ be an $m\times n$ orthonormal matrix and $\Omega$ an $m\times s$ SRFT matrix where the sketch size $s$ satisfies
\begin{equation*}
    4[\sqrt{n}+\sqrt{8\log(mn)}]^2\log n \leq s \leq m.
\end{equation*}
Then 
\begin{equation*}
    0.4\leq \sigma_n(\Omega^*U) \text{ and } \sigma_1(\Omega^*U) \leq 1.48
\end{equation*}
with failure probability at most $O(n^{-1})$.
\end{theorem}
Therefore as long as the sketch size is about $4n\log n$, SRFT matrices will approximately preserve the singular values of the original matrix with a reasonable failure probability. Unfortunately, the $\log n$ factor cannot be removed in general \cite{hmt}. One way to remove the $\log n$ factor is to replace the $R$ matrix in SRFT by a random 1-hashing matrix, giving us H-RHT \cite{cartis21}.  Fortunately, with the SRFT sketch, the sketch size $s = 2n$ for $n\geq 20$ does well at preserving the length of the original space in most applications \cite{hmt,nt21}. The cost for applying the SRFT matrix to an $m\times n$ matrix is $O(mn\log n )$ operations using the subsampled FFT algorithm \cite{r08}, but it is much easier to get $O(mn \log m)$ operations in practical implementations. This is still lower than the Gaussian matrix as long as $m$ is only polynomially larger than $n$, so the SRFT matrix is often used for practical reasons. There are also sparse sketching matrices which take advantage of the zero entries of the original matrix. An example is the CountSketch matrix \cite{cw13,woodruff14}. However, in this paper, we will focus on general matrices. Now we approach the null space problem using a sketch-and-solve method.
\begin{algorithm}
\caption{Solving the null space problem using sketch-and-solve}
\label{nsalg}
\begin{algorithmic}[1]
\Require{$A \in \mathbb{C}^{m\times n}$ with $m\gg n$, $s$ $(>n)$ the sketch size (rec. $s = 2n$), $k$ the number of right singular vectors desired}
\Ensure{$W\in \mathbb{C}^{n\times k}$ the trailing $k$ right singular vectors}
\State Draw a random sketch $S\in \mathbb{C}^{s\times m}$ with sketch size $s$
\State $SA$ = Sketch (e.g. SRFT) $A$ using $S$
\State $[U,\Sigma,V] = \text{SVD}(SA)$
\State $W = V(:,n-k+1:n)$, the trailing $k$ right singular vectors
\end{algorithmic}
\end{algorithm}

Algorithm \ref{nsalg} sketches the matrix $A$ and then computes the SVD of a smaller-sized matrix $SA$ to get the approximate right singular vectors corresponding to the smallest few singular values. This algorithm is not new and to our knowledge appeared first in \cite{exist}. 
This algorithm gives us a near-optimal solution with respect to the residual because the output obtained from Algorithm \ref{nsalg} is at most a modest constant larger than the optimal solution. This is made precise for the SRFT matrix in the following theorem; the sketch size $s=2n$ is not enough for theoretical guarantees, but works well in practice.
\begin{theorem}\label{thm:4timesres}
    Let $A\in \mathbb{C}^{m\times n}$ and $\tilde{A} = \Omega^* A \in \mathbb{C}^{s\times n}$ with $(m\geq s\geq n)$ where $\Omega \in \mathbb{C}^{m\times s}$ is the SRFT matrix with the sketch size $s$ satisfying 
    \begin{equation*}
    4[\sqrt{n}+\sqrt{8\log(mn)}]^2\log n \leq s \leq m.
\end{equation*} Then
    \begin{equation*}
         \norm{A\tilde{W}}_\text{F}  < 4\norm{AW}_\text{F}
    \end{equation*} with failure probability at most $O(n^{-1})$ where $\tilde{W}$ is the output from Algorithm \ref{nsalg} and $W$ is the exact solution to the null space problem (Problem \ref{problem}).
\end{theorem}
\begin{proof}
Let $A = U_A \Sigma_A V_A^*$ be a thin SVD of $A$ such that $U_A\in \mathbb{C}^{m\times n}$ and $\Sigma_A,V_A\in \mathbb{C}^{n\times n}$. Then by Ostrowski's theorem for singular values \cite{ostrowski} we have
\begin{equation*}
    \sigma_{\min}(\Omega^* U_A)\sigma_i (A) \leq \sigma_i (\Omega^* A) \leq \sigma_{\max}(\Omega^* U_A)\sigma_i (A)
\end{equation*} for $i = 1,2,...,n$ and
\begin{equation*}
\sigma_{\min}(\Omega^*U_A)\norm{A\tilde{W}}_\text{F} \leq \norm{\Omega^*A\tilde{W}}_\text{F}.
\end{equation*} Therefore, by Theorem \ref{srft}, with failure probability at most $O(n^{-1})$ we have
\begin{align*}
    \norm{A\tilde{W}}_\text{F} &\leq \frac{1}{\sigma_{\min}(\Omega^*U_A)}\norm{\Omega^*A\tilde{W}}_\text{F} \\ &= \frac{1}{\sigma_{\min}(\Omega^*U_A)} \sqrt{\sum_{j=1}^k\sigma_{n-j+1}^2(\Omega^*A)} \\ &\leq \frac{\sigma_{\max}(\Omega^*U_A)}{\sigma_{\min}(\Omega^*U_A)} \sqrt{\sum_{j=1}^k \sigma_{n-j+1}^2(A)} \\ &= \frac{1.48}{0.4}\norm{AW}_\text{F} <4\norm{AW}_\text{F}
\end{align*}
\end{proof}
\begin{remark}
A similar result can be obtained for other sketching matrices. In general, if the sketching matrix $S$ satisfies
\begin{equation*}
    (1-\delta)\sigma_i (A) \leq \sigma_i(SA) \leq (1+\delta) \sigma_i (A)
\end{equation*} for some $0 <\delta <1$ with high probability then
\begin{equation*}
    \norm{AW_S}_\text{F}  \leq \frac{1+\delta}{1-\delta}\norm{AW}_\text{F}
\end{equation*} with high probability where $W_S$ is the output of Algorithm \ref{nsalg} when using $S$ as the sketching matrix. (e.g. $\delta = \frac{1}{\sqrt{2}}$ for a Gaussian sketching matrix with the sketch size $s = 4n$)
\end{remark}

There is also a version of Algorithm \ref{nsalg} with $\epsilon$ tolerance rather than taking $k$ as input. This version finds all the singular values that are less than $\epsilon$ and set $W$ to be their corresponding right singular vectors. If we set $\epsilon = O(u)$ where $u$ is the unit roundoff, then we get the numerical null space.

For a sketch size $s=cn$ for a constant $c>0$, the complexity of Algorithm 1 is $O(mn\log n)$ for performing the sketch (line $2$) and $O(n^3)$ for calculating the trailing right singular vectors of $SA$ (line $3$). This gives us an overall complexity of $O(mn\log n+n^3)$ which is faster than the traditional methods, $O(mn^2)$. Now we look at the quality of the solution. We will examine how much the sketched solution (output $W$ of Algorithm \ref{nsalg}) deviates from the original solution (SVD of $A$).


\section{Accuracy of the sketch using multiplicative perturbation bounds}
\label{sec:RPT}
The standard way of quantifying the distance between two subspaces is to use the \emph{canonical angles} between subspaces.
\begin{definition}[Canonical angles {\cite[\S~I.5]{mpt}}]
Let $U,V\in \mathbb{C}^{n\times k}$ with $n\geq k$ be two matrices with orthonormal columns. Then the canonical angles between the subspaces spanned by $U$ and $V$ are $\{\theta_i\}_{i=1}^k$ where $\theta_i = \arccos(\sigma_i(U^* V))$, that is, the arccosine of the singular values of $U^*V$. We let $\Theta(U,V) = \text{diag}(\theta_1,...,\theta_k)$ and $\sin\Theta(U,V)$ and $\cos\Theta(U,V)$ be defined elementwise for the diagonal entries only.
\end{definition}
\begin{remark}\ 
\begin{enumerate}
    \item $\Theta(U,V) = \Theta(V,U)$ since $U^*V$ and $V^*U$ have the same singular values.
    \item By the CS decomposition (Chapter I Section 5, \cite{mpt}), the elements in the diagonal of $\sin\Theta(U,V)$ are the singular values of $U_\perp^*V$ or $U^*V_\perp$ where $U_\perp$ and $V_\perp$ have columns that give orthonormal bases to the orthogonal complements of range$(U)$ and range$(V)$ respectively. This implies that $\sin\Theta(U,V)$ and $\sin\Theta(U_\perp,V_\perp)$ have the same nonzero entries on their diagonal.
\end{enumerate}
\end{remark}

Now we look at the perturbation of right singular vectors using canonical angles. There are two different types of perturbations, namely additive and multiplicative perturbations. The results for additive perturbation was first proved in 1970 by Davis and Kahan \cite{dk70} for eigenvectors of symmetric matrices, and two years later Wedin \cite{wedin1972} derived analogous results for singular vectors. In this work, we will focus on multiplicative perturbations that give multiplicative perturbation bounds by Li \cite{li2}, as they arise naturally in our context, and give superior bounds in our setting. This type of bound in our setting has been studied in \cite{exist} for vectors, however there is no known study in our context for subspaces of dimension larger than $1$, which can be useful when we compute subspaces rather than vectors. 

\subsection{Computing subspaces of the same dimension}
Let $A\in \mathbb{C}^{m\times n}$ and $\tilde{A} \in \mathbb{C}^{s\times n}$ be matrices with $m\geq s\geq n$ and suppose that they have SVDs of the form
\begin{equation} \label{svdA}
    A = U\Sigma V^* = [U_1, U_2] 
    \begin{bmatrix}
\Sigma_1 & 0\\
0 & \Sigma_2
\end{bmatrix} [V_1, V_2]^*
\end{equation}
\begin{equation} \label{svdAtil}
    \tilde{A} = \tilde{U}\tilde{\Sigma}\tilde{V}^* = [\tilde{U}_1, \tilde{U}_2] 
    \begin{bmatrix}
\tilde{\Sigma}_1 & 0\\
0 & \tilde{\Sigma}_2
\end{bmatrix} [\tilde{V}_1, \tilde{V}_2]^*
\end{equation}
where $U\in \mathbb{C}^{m\times n}$, $\tilde{U}\in \mathbb{C}^{s\times n}$, $V,\tilde{V}\in \mathbb{C}^{n\times n}$ and the matrices with subscript $1$ have $(n-k)$ columns and subscript $2$ have $k$ columns with $k<n$ and 
\begin{equation} \label{singvalA}
    \Sigma_1 = \text{diag}(\sigma_1,...,\sigma_{n-k}), \hspace{0.5cm} \Sigma_2= \text{diag}(\sigma_{n-k+1},...,\sigma_n),
\end{equation}
\begin{equation} \label{singvalAtil}
    \tilde{\Sigma}_1 = \text{diag}(\tilde{\sigma}_1,...,\tilde{\sigma}_{n-k}), \hspace{0.5cm} \tilde{\Sigma}_2= \text{diag}(\tilde{\sigma}_{n-k+1},...,\tilde{\sigma}_n)
\end{equation}
where the singular values are ordered in non-increasing order.

Next we define a relative gap $\chi$. Let $a,b\in \mathbb{R}$ and define
\begin{equation*}
    \chi(a,b) = \frac{\abs{a-b}}{\sqrt{\abs{ab}}}
\end{equation*}
with the convention $0/0 = 0$ and $1/0 = \infty$. Note that $\chi$ is not a metric on $\mathbb{R}$ \cite{li1,li2}. We also recall a useful matrix norm inequality from \cite{mirsky1960}, which states that for any matrices $A,B$ and $C$ such that the product $ABC$ is defined, we have
\begin{equation*}
    \norm{ABC}\leq \norm{A}_2\norm{B}\norm{C}_2
\end{equation*} for any unitarily invariant norm $\norm{\cdot}$. Lastly, we review a key lemma from Li \cite{li2}.
\begin{lemma} \label{relsinlemma}
Let $A\in \mathbb{C}^{s\times s}$ and $B\in \mathbb{C}^{t\times t}$ be two positive semi-definite Hermitian matrices and let $E\in \mathbb{C}^{s\times t}$. Suppose that there exists $\alpha>0$ and $\delta>0$ such that 
\begin{equation*}
    \norm{A}_2\leq \alpha \text{ and }\norm{B^{-1}}_2^{-1}\geq \alpha+\delta
\end{equation*} 
or 
\begin{equation*}
     \norm{B}_2\leq \alpha \text{ and }\norm{A^{-1}}_2^{-1}\geq \alpha+\delta.
\end{equation*} 
Then the Sylvester equation $AX-XB = A^{1/2}EB^{1/2}$ has a unique solution $X\in\mathbb{C}^{s\times t}$, and moreover $\norm{X}\leq \norm{E}/\chi(\alpha,\alpha+\delta)$ for any unitarily invariant norm.
\end{lemma}

Now we prove a theorem that will be important for assessing the quality of the solution using the sketch-and-solve method. This is a modification of Theorem 4.8 in \cite{li2}. The modification allows multiplicative perturbation of an $m\times n$ matrix $A$ by a full row rank $s\times m$ rectangular matrix $X^*$ with $m\geq s\geq n$, whereas Li only considers non-singular square matrices. Our bound also tightens Li's bound by removing extra simplifications Li made in his proof. In our context, $A\in \mathbb{C}^{m\times n}$ corresponds to the original matrix and $X^*A \in \mathbb{C}^{s\times n}$ corresponds to the sketched matrix where $X$ is the sketching matrix.
\begin{theorem} \label{relsin}
Let $A\in \mathbb{C}^{m\times n}$ and $\tilde{A} = X^*A \in \mathbb{C}^{s\times n}$ ($m\geq s\geq n$) with thin SVDs as in Equations (\ref{svdA}, \ref{svdAtil}, \ref{singvalA}, \ref{singvalAtil}) where $X\in \mathbb{C}^{m \times s}$ is a matrix such that $X^*U$ has full rank. Let $X^*U=QR$ where $Q\in \mathbb{C}^{s\times n}$ and $R\in \mathbb{C}^{n\times n}$ be a thin QR factorization of $X^*U$ where $U$ is the orthonormal matrix from Equation \eqref{svdA}. Suppose that there exists $\alpha>0$ and $\delta>0$ such that 
\begin{equation*}
    \min_{1\leq i\leq n-k}\sigma_i \geq \alpha+\delta \hspace{1cm}\text{  and  }\hspace{1cm} \max_{1\leq j\leq k} \tilde{\sigma}_{n-k+j}\leq \alpha,
\end{equation*}
or 
\begin{equation*}
    \min_{1\leq i\leq n-k}\tilde{\sigma}_i \geq \alpha+\delta \hspace{1cm}\text{  and  }\hspace{1cm} \max_{1\leq j\leq k} \sigma_{n-k+j}\leq \alpha.
\end{equation*}
Then for any unitarily invariant norm we have
\begin{equation}\label{sinv2v2tilR}
    \norm{\sin\Theta(V_2,\tilde{V}_2)} \leq \frac{\norm{R-R^{-*}}}{\chi\big(\alpha^2,(\alpha+\delta)^2\big)}
\end{equation} where $V_2$ and $\tilde{V}_2$ are as in Equations (\ref{svdA},\ref{svdAtil}).
\end{theorem}

\begin{proof}
We define two matrices $B :=\Sigma V^*$ and $\tilde{B}:= R\Sigma V^*$. Notice that $A$ and $B$ have the same singular values and the right singular vectors. Also, note that since $X^*A = QR\Sigma V^*$, $\tilde{A}$ and $\tilde{B}$ have the same singular values and the right singular vectors, so to prove~\eqref{sinv2v2tilR} it suffices to work with $B$ and $\tilde{B}$.
Since $X$ is a full rank matrix, $R$ is invertible. We first note that
\begin{equation*}
    \tilde{B}^*\tilde{B} - B^*B = \tilde{B}^*RB-\tilde{B}^*R^{-*}B = \tilde{B}^*(R-R^{-*})B.
\end{equation*} This is equivalent to
\begin{align*}
    \tilde{V}\tilde{\Sigma}^2\tilde{V}^*-V\Sigma^2V^* = \tilde{V}\tilde{\Sigma}(Q^*\tilde{U})^*(R-R^{-*})\Sigma V^*,
\end{align*} since $\tilde{B} = Q^*\tilde{A} = Q^*\tilde{U}\tilde{\Sigma}\tilde{V}^*$. Right-multiplying $V$ and left-multiplying $\tilde{V}^*$ gives
\begin{align*}
    \tilde{\Sigma}^2\tilde{V}^*V-\tilde{V}^*V\Sigma^2 = \tilde{\Sigma}(Q^*\tilde{U})^*(R-R^{-*})\Sigma.
\end{align*}
Now define $M:= R-R^{-*}$ for shorthand then the above matrix equation can be represented as a $2\times 2$ block matrix as
\begin{align*}
    &\begin{bmatrix}
    \tilde{\Sigma}_1^2\tilde{V}_1^*V_1 -\tilde{V}_1^*V_1\Sigma_1^2 & \tilde{\Sigma}_1^2\tilde{V}_1^*V_2 - \tilde{V}_1^*V_2\Sigma_2^2 \\
    \tilde{\Sigma}_2^2\tilde{V}_2^*V_1 - \tilde{V}_2^*V_1\Sigma_1^2 & \tilde{\Sigma}_2^2\tilde{V}_2^*V_2 - \tilde{V}_2^*V_2\Sigma_2^2
    \end{bmatrix}  \\ &=
    \begin{bmatrix}
    \tilde{\Sigma}_1\tilde{U}_1^*QM
    \begin{bmatrix}
    \Sigma_1 \\ 0_{k\times (n-k)}
    \end{bmatrix} & \tilde{\Sigma}_1\tilde{U}_1^*QM
    \begin{bmatrix}
    0_{(n-k)\times k} \\ \Sigma_2
    \end{bmatrix}\\
    \tilde{\Sigma}_2\tilde{U}_2^*QM
    \begin{bmatrix}
    \Sigma_1 \\ 0_{k\times (n-k)}
    \end{bmatrix} & \tilde{\Sigma}_2\tilde{U}_2^*QM
    \begin{bmatrix}
    0_{(n-k)\times k} \\ \Sigma_2
    \end{bmatrix} 
    \end{bmatrix}  
\end{align*}
Taking the $(2,1)$-entry of the block matrix we get
\begin{equation*}
    \tilde{\Sigma}_2^2\tilde{V}_2^*V_1-\tilde{V}_2^*V_1\Sigma_1^2 = \tilde{\Sigma}_2\tilde{U}_2^*QM
    \begin{bmatrix}
    I_{n-k} \\ 0_{k\times (n-k)}
    \end{bmatrix}\Sigma_1.
\end{equation*}
This is in the form of the Sylvester equation in Lemma \ref{relsinlemma}. Thus, using Lemma \ref{relsinlemma} we get for any unitarily invariant norm
\begin{equation*}
    \norm{\tilde{V}_2^*V_1}\leq \frac{\norm{\tilde{U}_2^*QM
    \begin{bmatrix}
    I_{n-k} \\ 0_{k\times (n-k)}
    \end{bmatrix}}}{\chi\big(\alpha^2,(\alpha+\delta)^2\big)}.
\end{equation*}
Therefore
\begin{align*}
    \norm{\sin\Theta(V_2,\tilde{V}_2)} &=\norm{\tilde{V}_2^*V_1} \\ &\leq \frac{\norm{\tilde{U}_2^*QM
    \begin{bmatrix}
    I_{n-k} \\ 0_{k\times (n-k)}
    \end{bmatrix}}}{\chi\big(\alpha^2,(\alpha+\delta)^2\big)}\\
    &\leq \frac{\norm{\tilde{U}_2^*Q}_2\norm{M}\norm{\begin{bmatrix}
    I_{n-k} \\ 0_{k\times (n-k)}
    \end{bmatrix}}_2}{\chi\big(\alpha^2,(\alpha+\delta)^2\big)} \\
     &\leq \frac{\norm{M}}{\chi\big(\alpha^2,(\alpha+\delta)^2\big)} = \frac{\norm{R-R^{-*}}}{\chi\big(\alpha^2,(\alpha+\delta)^2\big)}
\end{align*} for any unitarily invariant norm.
\end{proof}
\begin{corollary} \label{gaussianbound}
In the setting of Theorem \ref{relsin}, we have
\begin{equation} \label{corsin}
    \norm{\sin\Theta(V_2,\tilde{V}_2)}_2 \leq \frac{2.1}{\chi(\alpha^2,(\alpha+\delta)^2)}
\end{equation} with failure probability at most $O(n^{-1})$ for an SRFT matrix $(X=\Omega)$ with the sketch size $s$ satisfying $4[\sqrt{n}+\sqrt{8\log(mn)}]^2\log n \leq s \leq m$. For the Gaussian case, that is, $X = G/\sqrt{s}$ where $G$ is a standard $m\times s$ Gaussian matrix, $\eqref{corsin}$ is satisfied with failure probability at most $\exp(-n/50)$ with the sketch size $s = 4n$.
\end{corollary}
\begin{proof}
The proof follows from the discussion on sketching matrices in Section \ref{sec:problem} and 
\begin{equation*}
    \norm{R-R^{-*}}_2 \leq \max_{\sigma\in [0.4,1.6]} \lvert \sigma-\sigma^{-1}\rvert = 2.1
\end{equation*} for $R$ as in Theorem \ref{relsin}.
\end{proof}

\subsubsection{A priori bound} \label{subsubsec: apriori}
In many cases, we are interested in a priori bounds for the accuracy of the trailing singular vectors using the sketch-and-solve method. We can obtain a priori bounds by substituting the singular values in place of $\alpha$ and $\delta$ in Corollary \ref{gaussianbound}. There are two a priori bounds, which are given below. If $\sigma_{n-k}>1.6 \sigma_{n-k+1}$ then
\begin{equation} \label{apriorib1}
    \norm{\sin\Theta(V_2,\tilde{V}_2)}_2 \leq \frac{2.1 \cdot \sigma_{n-k}\tilde{\sigma}_{n-k+1}}{\sigma_{n-k}^2-\tilde{\sigma}_{n-k+1}^2} \leq  \frac{3.36\cdot \sigma_{n-k}\sigma_{n-k+1}}{\sigma_{n-k}^2-2.56\cdot \sigma_{n-k+1}^2},
\end{equation}
and if $0.4 \sigma_{n-k} > \sigma_{n-k+1}$ then
\begin{equation} \label{apriorib2}
    \norm{\sin\Theta(V_2,\tilde{V}_2)}_2 \leq \frac{2.1 \tilde{\sigma}_{n-k}\sigma_{n-k+1}}{\tilde{\sigma}_{n-k}^2-\sigma_{n-k+1}^2} \leq  \frac{3.36\cdot \sigma_{n-k}\sigma_{n-k+1}}{0.16\cdot \sigma_{n-k}^2 - \sigma_{n-k+1}^2},
\end{equation}
since in the setting of Corollary \ref{gaussianbound}, we have
\begin{equation*}
    0.4\sigma_i \leq \tilde{\sigma}_i \leq 1.6\sigma_i
\end{equation*} for all $i$.

The upper bounds \eqref{apriorib1} and \eqref{apriorib2} are informative $(\ll 1)$ if $\sigma_{n-k} \gg \sigma_{n-k+1}$. In this case, the upper bounds are $\approx~\frac{\sigma_{n-k+1}}{\sigma_{n-k}}$, which are both much less than $1$. In particular if $\sigma_1 \geq \sigma_2 \geq \cdots \geq \sigma_{n-k} > \sigma_{n-k+1} = \cdots = \sigma_{n} = 0$ then the multiplicative perturbation by $X^*$ preserves the null space exactly as the upper bound becomes zero. In the presence of rounding errors, a backward stable solution would correspond to $\sigma_{n-k+1} = O(u)$ where $u$ is the unit roundoff. Assuming $\sigma_{n-k}$ is sufficiently larger than $u$, the last $k$ right singular vectors of the sketched matrix give an excellent approximation for the null space of the original matrix.

To illustrate the results, Figure \ref{Libound} shows the accuracy of the bound in Corollary~\ref{gaussianbound} for the case when $k=1$. We generated a random $1000 \times 100$ matrix with Haar distributed right singular vectors. The left singular vectors in the top left plot is also Haar distributed while the other $3$ plots were generated with the left singular vectors equal to $[I_{100},0]^T\in \mathbb{R}^{1000\times 100}$, which is a difficult (coherent) example for subspace embedding using SRFT~\cite{coherence, tropp11}.
 We set the singular values as $(\sigma_1,\sigma_2,...,\sigma_{n-1},\sigma_{n}) = (1,1,...,1,10^{-1},\sigma_n)$ where $\sigma_{n-1}/\sigma_n \in [10,10^{10}]$. We then sketch the matrix with a Gaussian matrix and an SRFT matrix. In Figure \ref{Libound}, we observe that except in the bottom left plot, the bounds in Corollary \ref{gaussianbound} give us the correct decay rate with a factor that is not too large. As seen in the bottom left plot, the SRFT sketch can fail, that is, the SRFT sketch fails to be a subspace embedding if we set the sketch size equal to $cn$ for a modest constant $c$, say $c=2,4$ for a difficult example (coherent) \cite{hmt, tropp11}. However, we can overcome this issue if we use the H-RHT sketch \cite{cartis21} or make the SRFT sketch size $\Theta(n\log n)$ as shown in the bottom right plot of Figure \ref{Libound}. This shows us that when $\sigma_{n-1}\gg \sigma_n$ we expect the sketched solution to give a very good approximation to the actual solution.
\begin{figure}[h] 
\centering
\hspace*{-1cm}
\includegraphics[scale = 0.5]{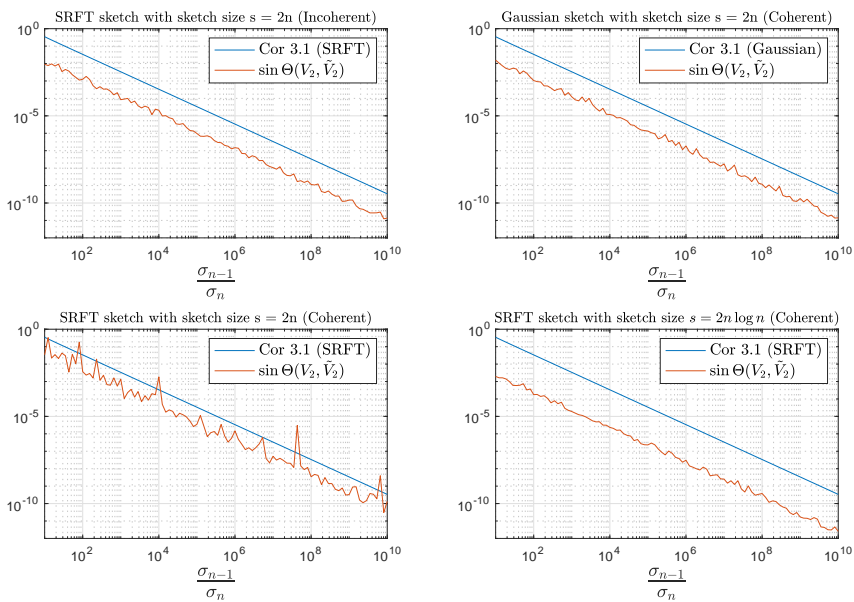}
\centering
\caption{Loglog plot of the bound in Corollary \ref{gaussianbound} against the actual sine values for  $m = 1000, n = 100$ and $k=1$. The matrix was generated with singular values diag$(1,1,...,1,10^{-1},\sigma_n)$ where $\sigma_{n-1}/\sigma_n \in [10,10^{10}]$. The bottom left plot shows a difficult example (coherent) for the SRFT sketch which can fail if the sketch size is not large enough. However, as seen on the bottom right plot it can be fixed by enlarging the sketch size to $\Theta(n\log n)$.
} 
\label{Libound}
\end{figure}

Up until now we have examined the canonical angles between two subspaces with the same dimension. We next extend Theorem \ref{relsin} to subspaces of different dimensions. The extension will be useful when the bound in Theorem \ref{relsin} is not useful, $O(1)$, and we want to search for a bigger subspace that can be shown to contain the subspace that we are looking for.

\subsection{Subspaces of different dimensions}
Let $A\in \mathbb{C}^{m\times n}$ and $\tilde{A} \in \mathbb{C}^{s\times n}$ be matrices with $m\geq s\geq n$ and suppose that they have SVDs of the form
\begin{equation} \label{svdA2}
    A = U\Sigma V^* = [U_1, U_2, U_3] 
    \begin{bmatrix}
\Sigma_1 & 0 & 0\\
0 & \Sigma_2 & 0 \\
0 & 0 & \Sigma_3
\end{bmatrix} [V_1, V_2, V_3]^*
\end{equation}
\begin{equation} \label{svdAtil2}
    \tilde{A} = \tilde{U}\tilde{\Sigma}\tilde{V}^* = [\tilde{U}_1, \tilde{U}_2,\tilde{U}_3] 
    \begin{bmatrix}
\tilde{\Sigma}_1 & 0 & 0 \\
0 & \tilde{\Sigma}_2 & 0 \\
0 & 0 & \tilde{\Sigma}_3
\end{bmatrix} [\tilde{V}_1, \tilde{V}_2,\tilde{V}_3]^*
\end{equation}
where $U\in \mathbb{C}^{m\times n}$, $\tilde{U}\in \mathbb{C}^{s\times n}$, $V,\tilde{V}\in \mathbb{C}^{n\times n}$ and the matrices with subscript $1$ have $(n-k)$ columns, those with subscript $2$ have $(k-\ell)$ columns and subscript $3$ have $\ell$ columns with $\ell <k<n$ and 
\begin{equation} \label{singvalA2}
    \Sigma_1 = \text{diag}(\sigma_1,...,\sigma_{n-k}), \hspace{0.0cm} \Sigma_2= \text{diag}(\sigma_{n-k+1},...,\sigma_{n-\ell}),
    \hspace{0.0cm} \Sigma_3= \text{diag}(\sigma_{n-\ell+1},...,\sigma_n),
\end{equation}
\begin{equation} \label{singvalAtil2}
    \tilde{\Sigma}_1 = \text{diag}(\tilde{\sigma}_1,...,\tilde{\sigma}_{n-k}), \hspace{0.0cm} \tilde{\Sigma}_2= \text{diag}(\tilde{\sigma}_{n-k+1},...,\tilde{\sigma}_{n-\ell}),
    \hspace{0.0cm} \tilde{\Sigma}_3= \text{diag}(\tilde{\sigma}_{n-\ell+1},...,\tilde{\sigma}_n)
\end{equation}
where the singular values are arranged in non-increasing order.

\begin{theorem} \label{relsin2}
Let $A\in \mathbb{C}^{m\times n}$ and $\tilde{A} = X^*A \in \mathbb{C}^{s\times n}$ ($m\geq s\geq n$) with thin SVDs as in Equations (\ref{svdA2}, \ref{svdAtil2}, \ref{singvalA2}, \ref{singvalAtil2}) where $X\in \mathbb{C}^{m \times s}$ is a matrix such that $X^*U$ has full rank. Let $X^*U=QR$ where $Q\in \mathbb{C}^{s\times n}$ and $R\in \mathbb{C}^{n\times n}$ be a thin QR factorization of $X^*U$ where $U$ is the orthonormal matrix from Equation \eqref{svdA2}. Suppose that there exists $\alpha>0$ and $\delta>0$ such that 
\begin{equation*}
    \min_{1\leq i\leq n-k}\sigma_i \geq \alpha+\delta \hspace{1cm}\text{  and  }\hspace{1cm} \max_{1\leq j\leq \ell} \tilde{\sigma}_{n-\ell+j}\leq \alpha.
\end{equation*}
Then for any unitarily invariant norm we have
\begin{equation*}
    \norm{\sin\Theta([V_2, V_3],\tilde{V}_3)} \leq \frac{\norm{R-R^{-*}}}{\chi\big(\alpha^2,(\alpha+\delta)^2\big)}.
\end{equation*}
Alternatively, if there exists $\alpha>0$ and $\delta>0$ such that 
\begin{equation*}
    \min_{1\leq i\leq n-k}\tilde{\sigma}_i \geq \alpha+\delta \hspace{1cm}\text{  and  }\hspace{1cm} \max_{1\leq j\leq \ell} \sigma_{n-\ell+j}\leq \alpha ,
\end{equation*}
then for any unitarily invariant norm we have
\begin{equation*}
    \norm{\sin\Theta(V_3,[\tilde{V}_2, \tilde{V}_3])} \leq \frac{\norm{R-R^{-*}}}{\chi\big(\alpha^2,(\alpha+\delta)^2\big)}
\end{equation*} where $V_2,V_3,\tilde{V}_2$ and $\tilde{V}_3$ are as in Equations (\ref{svdA2},\ref{svdAtil2}).
\end{theorem}
\begin{proof}
We follow the proof of Theorem \ref{relsin}. We consider the case
\begin{equation*}
    \min_{1\leq i\leq n-k}\sigma_i \geq \alpha+\delta \hspace{1cm}\text{  and  }\hspace{1cm} \max_{1\leq j\leq \ell} \tilde{\sigma}_{n-\ell+j}\leq \alpha.
\end{equation*}
The other case follows similarly. As in Theorem \ref{relsin}, we define two matrices $B :=\Sigma V^*$ and $\tilde{B}:= R\Sigma V^*$. Notice that $A$ and $B$ have the same singular values and the right singular vectors. Also, note that since $X^*A = QR\Sigma V^*$, $\tilde{A}$ and $\tilde{B}$ have the same singular values and the right singular vectors, so it suffices to work with $B$ and $\tilde{B}$.
Since $X$ is a full rank matrix, $R$ is invertible. We first note that
\begin{equation*}
    \tilde{B}^*\tilde{B} - B^*B = \tilde{B}^*RB-\tilde{B}^*R^{-*}B = \tilde{B}^*(R-R^{-*})B.
\end{equation*} This is equivalent to
\begin{align*}
    \tilde{V}\tilde{\Sigma}^2\tilde{V}^*-V\Sigma^2V^* = \tilde{V}\tilde{\Sigma}(Q^*\tilde{U})^*(R-R^{-*})\Sigma V^*,
\end{align*} since $\tilde{B} = Q^*\tilde{A} = Q^*\tilde{U}\tilde{\Sigma}\tilde{V}^*$. Right-multiplying $V$ and left-multiplying $\tilde{V}^*$ gives
\begin{align*}
    \tilde{\Sigma}^2\tilde{V}^*V-\tilde{V}^*V\Sigma^2 = \tilde{\Sigma}(Q^*\tilde{U})^*(R-R^{-*})\Sigma.
\end{align*}
Now the above matrix equation can be represented as a $3\times 3$ block matrix as in Theorem \ref{relsin}.
Taking the $(3,1)$-entry of the block matrix we get
\begin{equation*}
    \tilde{\Sigma}_3^2\tilde{V}_3^*V_1-\tilde{V}_3^*V_1\Sigma_1^2 = \tilde{\Sigma}_3\tilde{U}_3^*Q(R-R^{-*})
    \begin{bmatrix}
    I_{n-k} \\ 0_{k\times (n-k)}
    \end{bmatrix}\Sigma_1.
\end{equation*}
This is again in the form of the Sylvester equation in Lemma \ref{relsinlemma}. Thus, using Lemma \ref{relsinlemma} we get for any unitarily invariant norm
\begin{equation*}
    \norm{\tilde{V}_3^*V_1}\leq \frac{\norm{\tilde{U}_3^*Q(R-R^{-*})
    \begin{bmatrix}
    I_{n-k} \\ 0_{k\times (n-k)}
    \end{bmatrix}}}{\chi\big(\alpha^2,(\alpha+\delta)^2\big)}.
\end{equation*}
Finally we get 
\begin{align*}
    \norm{\sin\Theta([V_2,V_3],\tilde{V}_3)} &=\norm{\tilde{V}_3^*V_1} \\ &\leq \frac{\norm{\tilde{U}_3^*Q(R-R^{-*})
    \begin{bmatrix}
    I_{n-k} \\ 0_{k\times (n-k)}
    \end{bmatrix}}}{\chi\big(\alpha^2,(\alpha+\delta)^2\big)}\\
    &\leq \frac{\norm{\tilde{U}_3^*Q}_2\norm{R-R^{-*}}\norm{\begin{bmatrix}
    I_{n-k} \\ 0_{k\times (n-k)}
    \end{bmatrix}}_2}{\chi\big(\alpha^2,(\alpha+\delta)^2\big)} \\
     &\leq \frac{\norm{R-R^{-*}}}{\chi\big(\alpha^2,(\alpha+\delta)^2\big)}
\end{align*} for any unitarily invariant norm.
\end{proof}
\begin{corollary} \label{gaussianbound2}
In the setting of Theorem \ref{relsin2}, we have
\begin{equation} \label{corsin2}
    \norm{\sin\Theta(V_3,[\tilde{V}_2,\tilde{V}_3])}_2 \leq \frac{2.1}{\chi(\alpha^2,(\alpha+\delta)^2)}
\end{equation} with failure probability at most $O(n^{-1})$ for an SRFT matrix $(X=\Omega)$ with the sketch size $s$ satisfying $4[\sqrt{n}+\sqrt{8\log(mn)}]^2\log n \leq s \leq m$. For the Gaussian case, that is, $X = G/\sqrt{s}$ where $G$ is a standard $m\times s$ Gaussian matrix, $\eqref{corsin2}$ is satisfied with failure probability at most $\exp(-n/50)$ with the sketch size $s = 4n$.
\end{corollary}

\subsubsection{A priori bound}
Similarly as before, we may be interested in a priori bounds for the accuracy of the trailing singular vectors using the sketch-and-solve method when the subspaces sizes are different. We can obtain a priori bounds by substituting the singular values in place of $\alpha$ and $\delta$ in Corollary \ref{gaussianbound2}. There are two a priori bounds, which are given below. If $0.4\sigma_{n-k}> \sigma_{n-\ell+1}$ then
\begin{equation} \label{apriorib3}
    \norm{\sin\Theta(V_3,[\tilde{V}_2,\tilde{V}_3])}_2 \leq \frac{2.1 \cdot \tilde{\sigma}_{n-k}\sigma_{n-\ell+1}}{\tilde{\sigma}_{n-k}^2-\sigma_{n-\ell+1}^2} \leq \frac{3.36\cdot \sigma_{n-k}\sigma_{n-\ell+1}}{0.16\cdot \sigma_{n-k}^2-\sigma_{n-\ell+1}^2},
\end{equation}
and if $ \sigma_{n-k} > 1.6\sigma_{n-\ell+1}$ then
\begin{equation} \label{apriorib4}
    \norm{\sin\Theta([V_2,V_3],\tilde{V}_3)}_2 \leq
    \frac{2.1 \cdot \sigma_{n-k}\tilde{\sigma}_{n-\ell+1}}{\sigma_{n-k}^2-\tilde{\sigma}_{n-\ell+1}^2} \leq \frac{3.36\cdot \sigma_{n-k}\sigma_{n-\ell+1}}{\sigma_{n-k}^2-2.56\cdot \sigma_{n-\ell+1}^2}, 
\end{equation}
since in the setting of Corollary \ref{gaussianbound2}, we have
\begin{equation*}
    0.4\sigma_i \leq \tilde{\sigma}_i \leq 1.6\sigma_i
\end{equation*} for all $i$.

The upper bounds \eqref{apriorib3} and \eqref{apriorib4} are informative $(\ll 1)$ if $\sigma_{n-k} \gg \sigma_{n-\ell+1}$. In this case, the upper bounds are $\approx~\frac{\sigma_{n-\ell+1}}{\sigma_{n-k}}$, which are both much less than $1$.

Now we illustrate the results in Figure \ref{Libound2}. Figure \ref{Libound2} shows the accuracy of the bound in Corollary \ref{gaussianbound2} for the SRFT sketch in the case when $\ell = 1$ and $k = 2,...,n-1$. We generated a random $1000 
\times 100$ matrix by sampling the left and the right singular vectors as before for the coherent example, but the singular values now decay geometrically from $1$ to $10^{-10}$ for the left plot and the right plot has singular values equal to $1$ for the first $80$ singular values and $10^{-10}$ for the last $20$ singular values. In both of these cases, we have $\sigma_{n}/\sigma_{n-1} \approx 1$, so the previous bound in Corollary~\ref{gaussianbound} is useless. However, the bound in Corollary~\ref{gaussianbound2} becomes meaningful as we increase the subspace dimension $k$. 
 In Figure \ref{Libound2}, we see that the bounds in Corollary \ref{gaussianbound2} give us the correct decay rate with a modest factor for the left plot. For the right plot, we see that after the subspace becomes large enough ($>20$) to contain the right singular subspace corresponding the the singular value $10^{-10}$, we get a good subspace which approximately contains the last right singular vector of the original matrix.  This illustrates us that even when $\sigma_{n}/\sigma_{n-1} \approx 1$, as long as $\sigma_{n-k}\gg \sigma_n$ we can find a larger subspace of dimension $k$ that is expected to contain the right singular vector corresponding to the smallest singular value. 
\begin{figure}[h] 
\centering
\hspace*{-1.2cm}
\includegraphics[scale = 0.48]{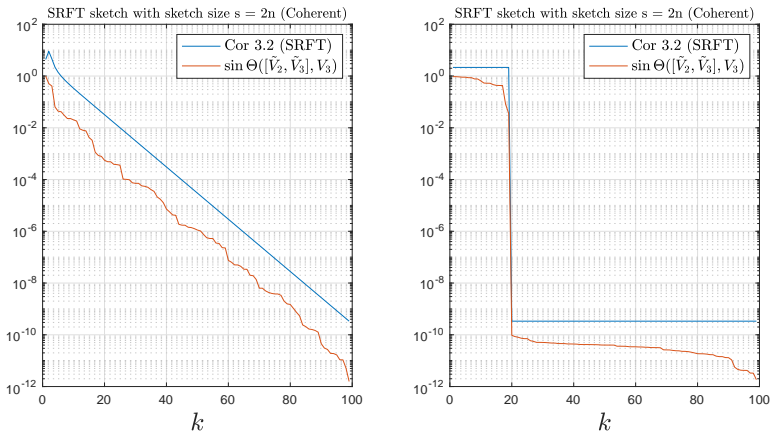}
\centering
\caption{Semilogy plot of the bound in Corollary \ref{gaussianbound2} against the actual sine values for $m = 1000, n = 100, \ell=1$ and subspace size $k = 2,...,(n-1)$. The matrix was generated with singular values that decay geometrically from $1$ to $10^{-10}$ for the left plot and the right plot has singular values equal to $1$ for the first $80$ singular values and the last $20$ singular values are equal to $10^{-10}$. The bound in Corollary \ref{gaussianbound2} gives us a good indication as to how much the computed subspace contains the last right singular vector.
}
\label{Libound2}
\end{figure}

 We now turn to the total least squares problem where we apply the ideas from Sections \ref{sec:problem} and \ref{sec:RPT}.
\section{Total Least Squares} \label{sec:tls}
Total least squares (TLS) \cite{tls80} is a type of least squares problem where errors in both independent and dependent variables are allowed. TLS is also known as errors-in-variables models in statistics. This is different from the standard least-squares problem where errors only occur in dependent variables. In a standard least-squares problem, we are interested in solving the following optimization problem
\begin{equation*}
    \min_{x\in \mathbb{R}^n} \norm{Ax-b}_2
\end{equation*} where $A\in \mathbb{R}^{m\times n}$ and $b\in \mathbb{R}^m$ with $m\geq n$. The errors only occur in the dependent variables, $b$, so the problem can be restated as
\begin{equation*}
    \min_{b+r \in \text{range}(A)} \norm{r}_2.
\end{equation*}
Now if we also allow errors in the independent variables, represented by $A$, we get the optimization problem formulation for the TLS problem, which is
\begin{equation*}
    \min_{b+r \in \text{range}(A+E)} \norm{[E\vert r]}_\text{F}
\end{equation*} where $[\cdot \vert \cdot]$ represents the augmented matrix of size $m\times (n+1)$. 

Now, if we allow multiple right-hand sides, say $k$, we get
\begin{equation} \label{tls}
    \min_{B+R \in \text{range}(A+E)} \norm{[E\vert R]}_\text{F}
\end{equation} where $A,E \in \mathbb{R}^{m\times n}$ and $B,R \in \mathbb{R}^{m\times k}$. We will also impose $m\geq n+k$ and $n\geq k$. If $[E_0\vert R_0]$ solves the optimization problem \eqref{tls}, then any $X\in \mathbb{R}^{n \times k}$ satisfying $(A+E_0)X = B_0+R_0$ is called the TLS solution.

In 1980, Golub and Van Loan \cite{tls80} gave a solution to the TLS problem. The algorithm solves the null space problem
\begin{equation*}
    V_k = \argmin_{V^*V = I_k} \norm{[A\vert B]V}_\text{F}
\end{equation*}  
with $V_k = \begin{bmatrix} V_{k,1} \\ V_{k,2} \end{bmatrix}$ where $V_{k,1}\in \mathbb{R}^{n\times k}$ and $V_{k,2}\in \mathbb{R}^{k\times k}$ and sets the solution to $X= -V_{k,1}V_{k,2}^{-1} \in \mathbb{R}^{n\times k}$. Note that since we are inverting $V_{k,2}$, the solution may not exist. There are known conditions for existence and uniqueness, for example when $\sigma_n(A) > \sigma_{n+1}([A\vert B])$. For more information see \cite{tls80,tls91}. The cost of computing the TLS solution is $O(m(n+k)^2)$ flops for computing the SVD of $[A\vert B] \in\mathbb{R}^{m\times (n+k)}$.

For the sketched version, we sketch the augmented matrix $[A\vert B]$ in $O(m(n+k)\log(n+k))$ flops and solve the TLS problem for a smaller-sized matrix using $O((n+k)^3)$ flops. Overall, the sketched version costs $O(m(n+k)\log(n+k)+(n+k)^3)$ flops, which becomes effective when $m \gg (n+k)$. 

For the accuracy of the sketched TLS solution, since the condition number of the TLS problem depends on how large the gap between $\sigma_n(A)$ and $\sigma_{n+1}([A\vert B])$ is \cite{tls80}, we expect the relative error $\frac{\norm{X-\tilde{X}}_2}{\norm{X}_2}$ and $\norm{\sin\Theta(X,\tilde{X})}_2$ to be proportional to the relative gap $\frac{\sigma_{n+1}([A\vert B])}{\sigma_n(A)}$ (Section \ref{sec:RPT}); here $X \in \mathbb{R}^{n\times k}$ is the TLS solution and $\tilde{X} \in \mathbb{R}^{n\times k}$ is the sketched TLS solution.
In addition, we also expect $\norm{\sin\Theta(V_k,\tilde{V}_k)}_2$ to be proportional to the relative gap $\frac{\sigma_{n+1}([A\vert B])}{\sigma_n(A)}$, 
where $V_k \in \mathbb{R}^{(n+k)\times k}$ is the trailing right singular vectors of the original TLS problem and $\tilde{V}_k \in \mathbb{R}^{(n+k)\times k}$ is the trailing right singular vectors of the sketched TLS problem. We demonstrate the relationship between the relative error $\frac{\norm{X-\tilde{X}}_2}{\norm{X}_2}$, $\norm{\sin\Theta(X,\tilde{X})}_2$ and $\norm{\sin\Theta(V_k,\tilde{V}_k)}_2$ with an experiment by varying the relative gap $\frac{\sigma_{n+1}([A\vert B])}{\sigma_n(A)}$.

Figure \ref{tlsplot} was generated using $A\in \mathbb{R}^{m\times n}$ and $B\in \mathbb{R}^{m\times k}$ with $m = 10^{5}$, $n = 10^{3}$ and $k = 10$ where $A$ is as in the previous experiments with the singular values that decay geometrically from $1$ to $10^{-3}$ and $B$ is a sum of a matrix in the span of $A$ with the same norm as $A$ and a Gaussian noise matrix of varying noise level between $10^{-3}$ and $10^{-12}$. In this experiment, since $A$ and $B$ are real matrices, the FFT matrix in the SRFT sketch was replaced with the DCT matrix. The sketch size was $s = 2020$, which is $2$ times the sum of the number of columns of $A$ and $B$.

\begin{figure}[h]
\hspace*{-1cm}
\includegraphics[scale = 0.45]{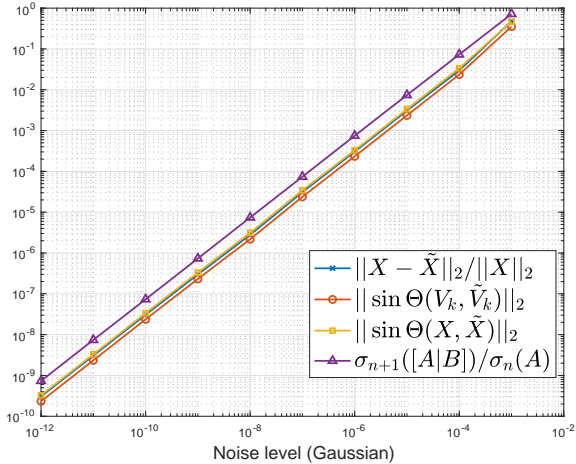}
\centering
\caption{A plot demonstrating the similarity between three different error metrics: relative error $\frac{\norm{X-\tilde{X}}_2}{\norm{X}_2}$ and the angles $\norm{\sin\Theta(X,\tilde{X})}_2$ and $\norm{\sin\Theta(V_k,\tilde{V}_k)}_2$. As we increase the relative gap $\frac{\sigma_{n+1}([A\vert B])}{\sigma_n(A)}$, we observe that all three error metrics behave similarly and have similar magnitudes.
}
\label{tlsplot}
\end{figure}

We observe in Figure \ref{tlsplot} that as we vary the noise level and hence the relative gap $\frac{\sigma_{n+1}([A\vert B])}{\sigma_n(A)}$, the error metrics $\frac{\norm{X-\tilde{X}}_2}{\norm{X}_2}$, $\norm{\sin\Theta(X,\tilde{X})}_2$ and $\norm{\sin\Theta(V_k,\tilde{V}_k)}_2$ all behave similarly. All three error metrics are proportional to the relative gap and they have similar magnitudes. In view of the theory developed in Section \ref{sec:RPT}, we use $\norm{\sin\Theta(V_k,\tilde{V}_k)}_2$ as an error metric for analyzing the accuracy of the sketched TLS solution. From Section \ref{subsubsec: apriori}, we have an a priori bound 
\begin{equation*}
    \norm{\sin\Theta(V_k,\tilde{V}_k)}_2 \lesssim O\left(\frac{\sigma_{n+1}([A\vert B])}{\sigma_n(A)}\right).
\end{equation*}
Now we demonstrate the speedup obtained using sketching (Algorithm \ref{nsalg}) with an experiment\footnote{All experiments were performed in MATLAB version 2021a on a workstation with Intel® Xeon® Silver 4314 CPU @ 2.40GHz ($16$ cores) and 512GB memory. The code for reproducing the experiments in Sections \ref{sec:tls} and \ref{sec:app} are available at \url{https://github.com/tpark4466/FastApproxNullSpace}.\label{fn:machine}}. 

\begin{table}[h]
\begin{center}
\begin{minipage}{\textwidth}
\caption{Total least squares problem with $A\in \mathbb{R}^{m\times 1000}$ and $B\in \mathbb{R}^{m\times 10}$ $(n = 1000, k = 10)$. $\tilde{X}\in \mathbb{R}^{1000\times 10}$ is the sketched TLS solution with the TLS error $\lVert[\tilde{E}\vert \tilde{R}]\rVert_{\text{F}}$ and $X\in \mathbb{R}^{1000\times 10}$ is the original TLS solution with the TLS error $\norm{[E\vert R]}_{\text{F}}$. The associated relative residual is $\frac{\lVert[\tilde{E}\vert \tilde{R}]\rVert_{\text{F}}}{\norm{[E\vert R]}_{\text{F}}}$ and the associated relative error is $\frac{\norm{X-\tilde{X}}_2}{\norm{X}_2}$. $V_k$ and $\tilde{V}_k$ are the $k$ trailing right singular vectors of the original TLS problem and the sketched TLS problem respectively.
}
\label{tls_table}
\begin{tabular*}{\textwidth}
{@{\extracolsep{\fill}}cccccc@{\extracolsep{\fill}}}
\toprule
$m$ & Speedup & $\norm{[E\vert R]}_{\text{F}}$ & $ \frac{\lVert[\tilde{E}\vert \tilde{R}]\rVert_{\text{F}}}{\norm{[E\vert R]}_{\text{F}}}$ & $ \frac{\norm{X-\tilde{X}}_2}{\norm{X}_2}$ & $\norm{\sin\Theta(V_k,\tilde{V}_k)}_2$ \\
\midrule
$2^{14}$ & $3.71 \times$ & $2.16 \cdot 10^{-8}$ & $1.39$ & $2.65 \cdot 10^{-6}$ & $2.21 \cdot 10^{-6}$ \\
$2^{15}$ & $6.07 \times$ & $2.21 \cdot 10^{-8}$ & $1.40$ & $2.98 \cdot 10^{-6}$ & $2.46 \cdot 10^{-6}$ \\
$2^{16}$ & $8.55 \times$ & $2.22 \cdot 10^{-8}$ & $1.40$ & $3.00 \cdot 10^{-6}$ & $2.32 \cdot 10^{-6}$ \\
$2^{17}$ & $14.12 \times$ & $2.22 \cdot 10^{-8}$ & $1.40$ & $2.89 \cdot 10^{-6}$ & $2.39 \cdot 10^{-6}$ \\
$2^{18}$ & $16.10 \times$ & $2.22 \cdot 10^{-8}$ & $1.41$ & $2.91 \cdot 10^{-6}$ & $2.24 \cdot 10^{-6}$ \\
\botrule
\end{tabular*}
\end{minipage}
\end{center}
\end{table}

Table \ref{tls_table} was generated using the same setup as in Figure \ref{tlsplot} with $A\in \mathbb{R}^{m\times 1000}$ and $B\in \mathbb{R}^{m\times 10}$, but with varying values of $m = 2^{14},2^{15},$ $2^{16}, 2^{17}, 2^{18}$ and a fixed Gaussian noise level of $10^{-3}$ for the matrix $B$.
In Table \ref{tls_table}, as we increase the value of $m$, we observe up to $16\times$ speedup, demonstrating the benefit of sketching. We also notice that the relative error and the sine of the angle between the trailing right singular vectors of the original TLS problem and the sketched TLS problem are small\footnote{The accuracy of the sketched solution can be improved by enlarging the sketch size at the cost of increasing the overall complexity.} and similar in magnitude, which is consistent with Figure \ref{tlsplot}. Lastly, we see that the sketched solution residual error is only a modest constant larger than the original error; this is expected since sketching gives a near-optimal solution (Section \ref{sec:problem}). Therefore, the sketched solution is a good approximate solution for the TLS problem.

\section{Updating and Downdating}
\label{sec:UD}
We now look at how row/column updates or downdates of the original matrix influence the sketch. Let $A \in \mathbb{C}^{m\times n}$ with $m\geq n$. We have been sketching from the left to get $SA\in \mathbb{C}^{s\times n}$ where $S\in \mathbb{C}^{s\times m}$ is a sketching matrix with sketch size $s$. Sometimes, in problems such as the AAA algorithm \cite{aaa} which we discuss in Section \ref{sec:app}, we want to update $A$, either by adding or removing a column and/or adding or removing a row. The null space of the updated matrix is then sought.

In this section, we devise a strategy in a similar spirit to the ones used for data streaming \cite{cw09,troppstream} to efficiently update the sketch $SA$ rather than sketching $A$ from scratch whenever a column or a row update/downdate to $A$ is made, so that the solution for the null space problem can be updated efficiently.
These strategies which will be shown below are possible because a particular realization of a sketching matrix can be reused for updates/downdates as long as the update/downdate does not depend on that realization of the sketching matrix. More specifically, if the update/downdate of the original matrix does not depend on the sketching matrix then the sketching matrix will only fail to be a subspace embedding for the updated/downdated matrix with probability that is at most a sum of exponentially small terms using the union bound. This happens because a single realization of the sketching matrix fails to be a subspace embedding for any matrix with exponentially small probability. (See Section \ref{sec:problem})

First, adding or removing a column is straightforward. If we add or remove a column from the original matrix then we can do the same for the sketch. 

By contrast, adding or removing a row is not so simple. For simplicity, we focus on row updates and downdates using the Gaussian sketching matrix, that is, $S = G/\sqrt{s} \in \mathbb{R}^{s\times m}$ where $G$ has entries that are independent standard normal random variables. We first observe that regardless of how many row updates and downdates are being done to the original matrix, the sketch $SA$ will always be an $s\times n$ matrix.\footnote{One caution is that if too many row downdates are done to $A$ so that $m\approx s$ then sketching becomes pointless.} 

\subsection{Row updating} Let us consider a row update first. Without loss of generality, suppose that a row $a_{m+1} \in \mathbb{C}^{1\times n}$ is added to the bottom of $A$ and let $A_{m+1} := \begin{bmatrix} A \\ a_{m+1}\end{bmatrix} \in \mathbb{C}^{(m+1)\times n}$. If we did not have the sketch of $A$ and if we had to sketch from scratch then we need to draw a Gaussian sketching matrix $\tilde{G}/\sqrt{s}\in \mathbb{R}^{s\times (m+1)}$ and left-multiply it to $A_{m+1}$. Now we find an equivalent process by reusing the sketch $SA\in\mathbb{R}^{s\times m}$. Since a row is appended to the end of $A$, we add a column to the end of $S$ giving us $S_{m+1} = [ G \vert g_{m+1}]/\sqrt{s}$ where $g_{m+1}\in \mathbb{R}^{s}$ is a Gaussian vector independent of $G$ and $A$. Notice that $[G\vert g_{m+1}]$ and $\tilde{G}$ are equal in distribution. Therefore, the updated sketch for $A_{m+1}$ becomes
\begin{equation*}
    S_{m+1} A_{m+1} = \frac{1}{\sqrt{s}}[G\vert g_{m+1}] \begin{bmatrix} A \\ a_{m+1}\end{bmatrix} = SA+\frac{1}{\sqrt{s}}g_{m+1}a_{m+1}
\end{equation*} with the updated sketching matrix $S_{m+1} = [G\vert g_{m+1}]/\sqrt{s}$. Algorithm \ref{rowupdate} shows this.

\begin{algorithm}
\caption{A row update (adding the $(m+1)$th row to $A\in \mathbb{C}^{m\times n}$)}
\label{rowupdate}
\begin{algorithmic}[1]
\Require{$S\in \mathbb{R}^{s\times m}$ Sketching matrix (sketch size $s$ ($m\gg s> n$)), $a_{m+1}~\in~\mathbb{C}^{1\times n}$ the row being added, $SA\in \mathbb{C}^{s\times n}$ the sketch of $A$}
\Ensure{$S_{m+1}\in \mathbb{R}^{s\times (m+1)}$ updated sketching matrix, $S_{m+1}A_{m+1}\in \mathbb{C}^{s\times n}$ updated sketch}
\State Sample $g_{m+1}\in \mathbb{R}^s$ with i.i.d. $N(0,1)$ random variable, also independent of $S$
\State Set $S_r= [S\vert g_{m+1}/\sqrt{s}] $
\State Set $S_rA_r = SA + \frac{g_{m+1}a_{m+1}}{\sqrt{s}}$
\end{algorithmic}
\end{algorithm}

\subsection{Row downdating} Row downdating is similar to a row update. Let $A_j\in \mathbb{C}^{(m-1)\times n}$ be the matrix with the $j$th row removed from $A$. Using a similar idea as a row update, we let $S_j\in \mathbb{R}^{s\times (m-1)}$ be the matrix with the $j$th column removed from $S$. Then the updated sketch for $A_{j}$ becomes
\begin{equation*}
S_j A_j = \left[S(:,1:j-1), S(:,j+1:m)\right] \begin{bmatrix} A(1:j-1,:) \\ A(j+1:m,:)\end{bmatrix} = SA-S(:,j)A(j,:)
\end{equation*} using MATLAB notation. Thus, the updated sketching matrix becomes $S_j$ and the updated sketch becomes $SA-S(:,j)A(j,:)$. Algorithm \ref{rowdowndate} shows this.

\begin{algorithm}
\caption{A row downdate (removing the $j$th row from $A\in \mathbb{C}^{m\times n}$)}
\label{rowdowndate}
\begin{algorithmic}[1]
\Require{$S\in \mathbb{R}^{s\times m}$ Sketching matrix (sketch size $s$ ($m\gg s> n$)), $SA~\in~\mathbb{C}^{s\times n}$ the sketch of $A$}
\Ensure{$S_j\in \mathbb{R}^{s\times (m-1)}$ downdated sketching matrix, $S_jA_j\in \mathbb{C}^{s\times n}$ downdated sketch}
\State Set $S_j= [S(:,1:j-1),S(:,j+1:m)]$, $S$ with $j$th column deleted
\State Set $S_jA_j = SA - S(:,j)A(j,:)$
\end{algorithmic}
\end{algorithm}

\subsection{Non-Gaussian sketch}
We have considered the Gaussian sketching matrix for row updates and downdates. For other sketching matrices such as the SRFT, the analysis is more difficult. However, in practice many sketching matrices behave similarly to the Gaussian sketching matrix and often Gaussian analysis reflects the performance in practice \cite{mt20}. For column updates and downdates, the strategy is the same for all sketching matrices, but not necessarily for rows. The strategy for row updates/downdates can be the same for certain classes of sketching matrices, for example, all the sketching matrices that have i.i.d. columns. The strategy we suggest for other sketching matrices for row updates and downdates is the following.

For a row update, we append a standard Gaussian column vector to the sketching matrix at the position where the new row gets appended to $A$. The updated sketch will then be $SA+g a_{m+1}/\sqrt{s}$ where $SA\in \mathbb{C}^{s\times n}$ is the previous sketch, $g\in \mathbb{C}^{s}$ is a standard Gaussian column vector independent with $S$ and $a_{m+1}\in \mathbb{C}^{1\times n}$ is the row appended to $A$. For a row downdate, we would need to remove the column from the sketching matrix corresponding to the row that will be removed from $A$. Removing this column is essentially the same as zeroing out its corresponding row in $A$ and keeping the original sketch. Therefore we propose the following. We take the standard basis vector corresponding to the index of the row that will be removed from $A$, say $\mathbf{e}_j$, sketch $\mathbf{e}_j$ using the original sketching matrix $S$ giving $S\mathbf{e}_j=:\tilde{\mathbf{e}}_j\in \mathbb{C}^s$ and subtract $\tilde{\mathbf{e}}_j A(j,:) \in \mathbb{C}^{s\times n}$ from the original sketch, where $A(j,:)$ is the row that is to be removed from $A$. This process removes the contribution from the removed row in the original sketch.

\subsection{Complexity of updating the sketch}
We discuss the complexity of reusing the sketch. We assume that the sketching matrix is an SRFT matrix with the sketch size $s = 2n$ and we are sketching the matrix $A\in \mathbb{C}^{m\times n}$ ($m\gg n$). For a column downdate, it is essentially free because we only need to remove a column from the sketch. For a column update, we need to sketch a column which costs $O(m\log m)$ flops. For a row update we need to do a column-row multiplication followed by an addition of two $s \times n$ matrices which cost an overall $O(sn)$ flops. For a row downdate, we need to sketch a standard basis vector which costs $O(m\log m)$ flops and perform a column-row multiplication followed by a subtraction of two $s \times n$ matrices which cost $O(sn)$ flops. Overall, a row downdate costs $O(m\log m +sn)$ flops. This is better than resketching, which costs $O(mn\log n)$ flops.  Therefore, whenever an update or a downdate is made to the original matrix, updating the sketch is at least about $O(n)$ times better than resketching the updated/downdated matrix from scratch. 

We now use the ideas in this section to speed up the AAA algorithm for rational approximation by reusing the sketch in the sketch-and-solve method.

\section{AAA algorithm for rational approximation}
\label{sec:app}
The goal of rational approximation is: given a (possibly complicated) function $\mathtt{f}:\mathbb{C}\rightarrow\mathbb{C}$, find a rational function $\mathtt{r}:\mathbb{C}\rightarrow\mathbb{C}$ that approximates $\mathtt{f}\approx \mathtt{r}$, in a domain $\Omega\subseteq \mathbb{C}$. 
The AAA algorithm \cite{aaa} is a powerful algorithm for this task, 
requiring only a set of distinct sample points, $z_1,z_2,...,z_m\in\Omega$, and their function values, $f_1,f_2,...,f_m$, that is, $f_i=\mathtt{f}(z_i)$. 
The flexibility and empirical near-optimality of AAA have resulted in its use in a large number of applications, including nonlinear eigenvalue problems~\cite{lietaert2022automatic}, conformal maps~\cite{gopal2019representation}, model order reduction~\cite{gosea2020aaa}, and signal processing~\cite{derevianko2023esprit}. 
Rational approximation 
 can significantly outperform the more standard polynomial approximation when e.g. $\Omega$ is unbounded or when $\mathtt{f}$ has singularities~\cite{trefethenatap}.

\subsection{A brief summary of the AAA algorithm}
Let us outline the AAA algorithm, emphasizing how it results in a null space problem. 
AAA has two key ingredients. The first is to 
use the barycentric representation of rational functions, instead of the standard quotient of polynomials:
\begin{equation*}
    \mathtt{r}(z) = \mathtt{n}(z)/\mathtt{d}(z) = \sum_{j=1}^n \frac{f_{i_j} w_j}{z-z_{i_j}}\Bigg/\sum_{j=1}^n \frac{w_j}{z-z_{i_j}}
\end{equation*} where $w_1,w_2,...,w_n$ are weights and $i_1,i_2,...,i_n \in \{1,2,...,m\}$ are distinct indices, where usually $m\gg n$. 
The second key ingredient is the greedy selection of the so-called \emph{support points} $z_{i_j}$, which are chosen as the points where the current error $\abs{ \mathtt{f}(z_i)-\mathtt{r}(z_i)}$ is maximized.
The algorithm 
then finds the minimizer of the linerized least-squares error $\|\mathtt{fd}-\mathtt{n}\|$, where the norm is the $\ell_2$ norm on the sample points. 
This is equivalent to 
forming a matrix called the Loewner matrix \cite{aaa}, $A^{(k)}$, whose $(i,j)$ entry is $\frac{f(z_i)-f(z_j)}{z_i-z_j}$, and finding $ \norm{Aw}_2=\|\mathtt{fd}-\mathtt{n}\|$ over unit-norm vectors $w$, i.e., a null space problem with $k=1$. 
As the algorithm iterates, a column is added (the degree of $\mathtt{r}$ is increased)
while a row is removed (a support point is added, hence removed from the least-squares equation) from the Loewner matrix from the previous iteration. Therefore at the $k$th iteration, $A^{(k)}\in \mathbb{C}^{(m-k)\times k}$ and the iteration is terminated once the tolerance is met. The precise details are covered in the original paper \cite{aaa}. 

To summarize, the main computational task in AAA is that at each iteration, say $k$, we solve for the weights $w\in \mathbb{C}^{k}$ that solve the following null space problem
\begin{equation*}
    \min_{\norm{w}_2=1}\norm{A^{(k)}w}_2.
\end{equation*} 
The algorithm computes the SVD of an $(m-k)\times k$ matrix at iteration $k$, giving us the overall complexity of $O(mn^3)$ for the AAA algorithm where $n$ is the degree of the rational approximant $\mathtt{r}$.

\subsection{Sketching and updating+downdating for speeding up AAA}
Given that the null space problem is the main computational task in AAA, it is natural to apply the algorithms in this paper to speed it up. The first straightforward idea is as follows:
At the $k$th iteration, when we solve for the right singular vector corresponding to the smallest singular value of $A^{(k)}$ we can sketch the matrix $A^{(k)}$ and then take the SVD of a smaller-sized matrix instead. This will reduce the complexity to $O(mn^2\log n+n^4)$ when $m\gg n$. However, in our implementation of SRFT using the standard FFT we achieve a theoretical flop count of $O(mk\log m+k^3)$ at iteration $k$. This limits the speedup we can obtain as $n$ is usually not too large ($\lesssim 200$) in most applications \cite{aaa}.

Fortunately, we can improve the speed further by noting that at each iteration of the AAA algorithm a column is added while a row is deleted from the Loewner matrix. 
Therefore we can use the strategies discussed in Section \ref{sec:UD}.\footnote{At each iteration, a column update and a row downdate to the Loewner matrix is the result of choosing a support point. Since we begin with the original matrix and the sketching matrix being independent, the failure probability for the column update and the row downdate concerning any independently chosen support point is exponentially small (Section \ref{sec:UD}) Therefore, regardless of how the sketch influences the choice of the support point, the sketch would succeed, in the worst case, with all but a sum of exponentially small probabilities.}. For deleting a row, we can use Algorithm \ref{rowdowndate} and for adding a column, we can sketch the new column and append it to the sketch. The overall sketch then becomes
\begin{equation*}
    \tilde{A}^{(k)}= \left[\tilde{A}^{(k-1)}-\tilde{\mathbf{e}}_j A^{(k-1)}(j,:), \tilde{A}_k\right] \in \mathbb{C}^{s\times k}
\end{equation*} at iteration $k$ where $\tilde{A}^{(k-1)}$ is the sketch of $A^{(k-1)}$, $\tilde{\mathbf{e}}_j$ is the sketch of the $j$th canonical vector, $A^{(k-1)}(j,:)$ is the deleted row and $\tilde{A}_k$ is the sketch of the added column. After the sketch is made, we take the SVD of the sketch $\tilde{A}^{(k)}$ and extract the last trailing singular vector from the SVD. The overall complexity is $O(mn\log m +n^4)$ using the SRFT sketch with the sketch size $s=2n$. Thus, when $m\gg n$ and $m$ is at most exponentially larger than $n$, we get a lower complexity than both the original AAA algorithm with $O(mn^3)$ flops, and the version where we resketch the entire Loewner matrix $A^{(k)}$ at each iteration, requiring $O(mn^2\log n +n^4)$ flops.

\subsubsection{Other approaches for speeding up AAA} In \cite{hochman2017fastaaa}, Hochman designs an algorithm to speed up the AAA algorithm based on Cholesky update/downdate of the Gram matrix of $A^{(k)}$. In the AAA algorithm, the Loewner matrix $A^{(k)}$ can become extremely ill-conditioned so Cholesky update/downdate can be numerically unstable \cite{downdate}. Also, the complexity of Hochman's algorithm is $O(mn^2)$ and our algorithm, with complexity $O(mn\log m +n^4)$, is faster as long as $m$ is larger than $n^2$. 

\subsection{Experiment}

\begin{table}[h]
\begin{center}
\begin{minipage}{\textwidth}
\caption{AAA speedup for four functions that give different values of $n$ with $m = 10^6$ points. The points were sampled uniformly at random from the domain of each function.}
\label{aaa_reuse}
\begin{tabular*}{\textwidth}
{@{\extracolsep{\fill}}cccc@{\extracolsep{\fill}}}
\toprule
$f(z)$ & $n$ & Speedup & Domain\\
\midrule
 $\log(2+z^4)/(1-16z^4)$ & 32  & 10.49$\times$ & $\{z:\abs{z}=1\}$ \\
 $\sqrt{z(1-z)}\sqrt{(z-i)(1+i-z)}$  & 60  & 14.00$\times$ & $\{z= x+iy: x,y\in [0,1]\}$ \\
 $\tan(128z)$  & 105  & 19.40$\times$ & $\{z: \abs{z}\leq 1 \}$ \\
 $\tan(256z)$  & 190 & 32.69$\times$ & $\{z: \abs{z} \leq 1 \}$ \\
 \botrule
\end{tabular*}
\end{minipage}
\end{center}
\end{table}

In Table \ref{aaa_reuse}, we conducted experiments 
with various functions using $10^6$ points randomly sampled\footnote{It is often excessive and unnecessary to take a million sample points in AAA. In many cases, say $10^3$--$10^4$ points would suffice. Nonetheless, when the function has singularities on or near the domain of interest, and the precise location of the singularity is unknown, it is sensible to take as many sample points as possible to ensure sufficiently many sample points are near the singularity.} from the domain of each function. These functions were chosen so that the functions give different values of $n$ in the rational approximation. This will demonstrate how the sketched version of the AAA algorithm performs in the high-dimensional case when compared with the original AAA algorithm. By reusing the sketch we already see a great speedup even for small values of $n$. The function $f(z)=\log(2+z^4)/(1-16z^4)$ is estimated by a rational function with $n=32$ and we achieved more than $10\times$ speedup. For a larger value of $n$, in the case $f(z)= \tan(256z)$ with $n=190$, we get more than $30\times$ speedup.

We conclude with a plot (Figure \ref{aaaplot}) showing the qualitative similarity between the rational approximations of $f(z) =\sqrt{z(1-z)}\sqrt{(z-i)(1+i-z)}$ obtained using the original AAA algorithm and the version where we have reused the sketch. Both the AAA approximant $r_1$ and the sketched AAA approximant $r_2$ provide good approximation to the original function $f$.
\begin{figure}[h]
\hspace*{-1cm}
\includegraphics[scale = 0.55]{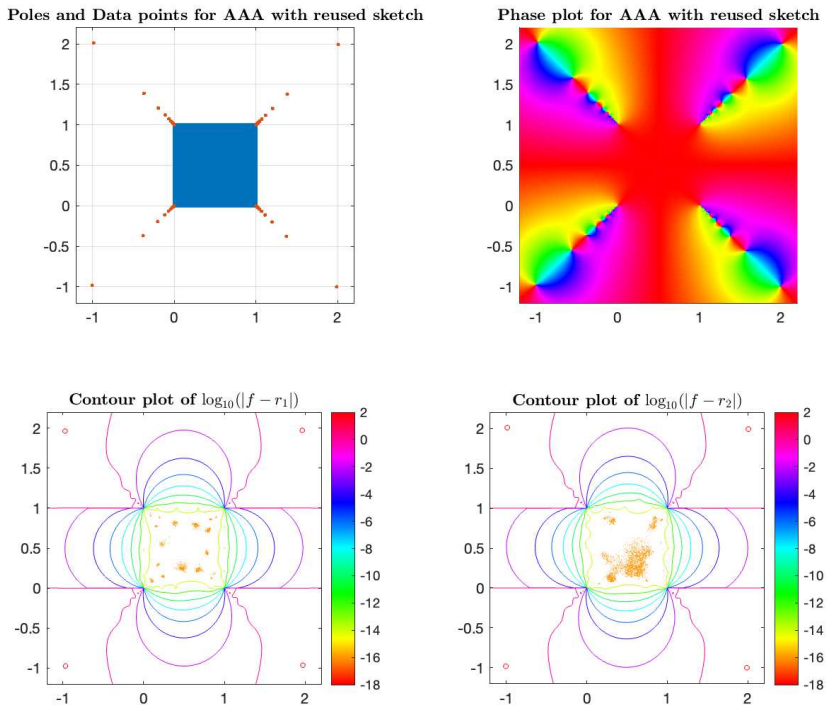}
\centering
\caption{Plots obtained with $f(z) =\sqrt{z(1-z)}\sqrt{(z-i)(1+i-z)}$ with $10^6$ points randomly sampled from $\{z= x+iy: x,y\in [0,1]\}$. $r_1$ and $r_2$ are respectively the rational approximation obtained using the original AAA algorithm and the AAA algorithm with reused sketch. The bottom two contour plots show the logarithm (base 10) of the approximation error. We see that the two versions give approximately the same quality of approximation. The top left plot shows the data points (blue) and the poles (red) for $r_2$ and the top right plot shows the phase plot of $r_2$.
}
\label{aaaplot}
\end{figure}

\bmhead{Acknowledgments}

TP was supported by the Heilbronn Institute for Mathematical Research.
We thank the anonymous reviewers for their many insightful comments and suggestions.

\section*{Declarations}

\bmhead{Conflict of interest} The authors declare no potential conflict of interests.

\bibliography{sn-bibliography}



\end{document}